\documentclass[12pt]{article}
\usepackage[margin=1in]{geometry} 
\usepackage{amsmath,amsthm,amsfonts}
\usepackage{physics}

\newtheorem{theorem}{Theorem}
\newtheorem{lemma}[theorem]{Lemma}
\newtheorem{definition}[theorem]{Definition}

\newtheorem{corollary}[theorem]{Corollary}

\newtheorem{remark}[theorem]{Remark}

\newcommand{\R}{\mathbb{R}}
\newcommand{\C}{\mathbb{C}}
\newcommand{\M}{\mathbb{M}}
\renewcommand{\H}{\mathbb{H}}

\renewcommand{\phi}{\varphi}
\renewcommand{\epsilon}{\varepsilon}

\begin{document}

\title{Clifford Analysis with Indefinite Signature}

\author{Matvei Libine\footnote{Department of Mathematics, Indiana University,
Rawles Hall, 831 East 3rd St, Bloomington, IN 47405}
and Ely Sandine\footnote{Undergraduate student at
Cornell University, Ithaca, NY 14850}}

\maketitle

\begin{abstract}
We extend constructions of classical Clifford analysis to the case of
indefinite non-degenerate quadratic forms.
We define $(p,q)$-left- and right-monogenic functions by means of
Dirac operators that factor a certain wave operator.
We prove two different versions of Cauchy's integral formulas for these
functions.
The two formulas arise from dealing with singularities in distinct ways,
and are inspired by the methods of \cite{L, FL}.
These results indicate the merit of these methods for dealing with
singularities.
\end{abstract}

\section{Introduction}
Many results of complex analysis have analogues in quaternionic analysis.
In particular, there are analogues of complex holomorphic functions called
(left- and right-) regular functions. Cauchy's integral formula for complex
holomorphic functions $f(z)$
\begin{equation*}
f(z_0) = \frac 1{2\pi i} \oint \frac {f(z)}{z-z_0}\,dz
\end{equation*}
extends to the quaternionic setting, and its analogues for (left- and right-)
regular functions are usually referred to as Cauchy-Fueter formulas.
For modern introductions to quaternionic analysis see, for example,
\cite{Su, CSSS}.

Complex numbers $\C$ and quaternions $\H$ are special cases of
Clifford algebras.
(For an elementary introduction to Clifford algebras see, for example,
\cite{G}.)
There is a further extension of complex and quaternionic analysis called
Clifford analysis.
For Clifford algebras associated to positive definite quadratic forms on real
vector spaces, Clifford analysis is very similar to complex and quaternionic
analysis (see, for example, \cite{BDS, DSS, GM} and references therein).
Furthermore, J.~Ryan has initiated the study of Clifford analysis in the
setting of complex Clifford algebras \cite{R1,R2}.

If a Clifford algebra is associated to a quadratic form that is not positive
definite, the function that should serve as the reproducing kernel in the
Cauchy type formula has singularities that always intersect the contour of
integration, thus rendering a potential reproducing integral formula
meaningless.
For this reason, most developments of real Clifford analysis so far
have been in the setting of Clifford algebras associated to positive
definite quadratic forms.

We develop Clifford analysis in the setting of Clifford algebras associated
to indefinite non-degenerate quadratic forms by treating the relevant
Clifford algebras as real subalgebras of complex Clifford algebras
and, in a certain sense, restricting the complex Dirac operators and
functions to the appropriate real subspaces.
This is done in complete parallel with K.~Imaeda's approach in \cite{I},
where he first extends classical quaternionic analysis to the algebra of
biquaternions $\H_{\C} = \H \otimes_{\R} \C$ and then ``restricts'' to
the Minkowski space $\M$ by realizing $\M$ as a real form of $\H_{\C}$.

Similarly, the works \cite{L, FL} extend classical quaternionic
analysis to split quaternions $\H_{\R}$ (also known as coquaternions).
The algebra $\H_{\R}$ is isomorphic to the algebra of real $2 \times 2$ matrices
and is another example of a Clifford algebra, this time associated to an
indefinite non-degenerate quadratic form on $\R^2$.
The issue of singularities intersecting the contour of integration is
resolved in two different ways, leading to two types of integral formulas.
The first method uses biquaternions $\H_{\C}$, which
contain the split quaternions $\H_{\R}$ as a real subalgebra.
By considering the holomorphic extension of a regular function from $\H_{\R}$
into $\H_{\C}$ and deforming the contour of integration so it no longer
intersects the singularities, one obtains the first version of Cauchy's
formulas for regular functions in $\H_{\R}$.
The second method involves inserting a purely imaginary term
$i\epsilon\norm{X-X_0}^2$ into the denominator of the reproducing kernel
to prevent singularities, then showing that the limit as $\epsilon \to 0$
exists and produces another version of Cauchy's formulas for regular
functions in $\H_{\R}$.

Methods for dealing with non-positive definite signature setting developed in
\cite{L,FL} were used in \cite{RCW} and \cite{P}.
Thus \cite{RCW} extends the results of classical Clifford analysis to the
setting of what the authors call split-quaternionic Hermitian Clifford analysis.
And \cite{P} extends analysis over octonions to split octonions.
Both papers have to deal with exactly the same issue of the set of
singularities intersecting the contour of integration, and this issue is
resolved in the same way as in \cite{L, FL}.
In this paper we use methods of \cite{L, FL} to extend constructions
of classical Clifford analysis to the case of indefinite non-degenerate
quadratic forms.
In a larger context, our results suggest that the methods of handling
singularities developed in \cite{L, FL} are quite general.

We begin by recalling in Section \ref{2} the construction of
the universal Clifford algebras. These constructions are done in both the real
and complex cases, and we regard the real Clifford algebra as a subalgebra
of the complex Clifford algebra.
Following this, in Section \ref{3} we define Dirac operators and introduce
left- and right-monogenic functions as functions annihilated by Dirac operators.
Both real and complex cases are considered.
The complex case was mostly developed by J.~Ryan \cite{R1,R2},
and the real case is, in a certain sense, the restriction of the
complex operators and functions to the real subspace.
In Section \ref{4} we introduce Clifford-valued differential forms $D_nz$,
$D_{p,q}x$ and establish some of their properties.
These forms appear in the statements of Cauchy's formulas for monogenic
functions.
We also state the classical Cauchy's integral formulas for left- and
right-monogenic functions in the positive definite case
(Theorem \ref{standardintegralformula}).
In Section \ref{5} we state and prove the first version of Cauchy's integral
formulas for left- and right-monogenic functions
(Theorem \ref{firstintegralformula}).
This version requires having holomorphic extension of the monogenic functions
to the complex space and utilizes a deformation of the contour of integration
to avoid the set of singularities.
In Section \ref{6} we prove the second version of Cauchy's integral formulas
for left- and right-monogenic functions (Theorem \ref{secondintegralformula}).
This version involves inserting a purely imaginary term
$i\epsilon\norm{X-X_0}^2$ into the denominator of the reproducing kernel
to prevent singularities, then taking limit as $\epsilon \to 0^+$.
We first establish Theorem \ref{secondintegralformula} up to a constant
coefficient \eqref{2nd-formula-constant}. Then we pin down the value of the
coefficient using the first version of integral formulas
(Theorem \ref{firstintegralformula}), which is already established
(Lemma \ref{C-lemma}).

\thanks{
This research was made possible by the Indiana University, Bloomington,
Math REU (research experiences for undergraduates) program,
funded by NSF Award $\#1757857$.
We would like to thank Prof. Chris Connell for organizing and facilitating
this wonderful program.
We would also like to thank Ms. Mandie McCarthy for her administrative work,
the various professors who gave talks, and the other REU students for their
company.}

\section{Clifford Algebras}  \label{2}

In this section we recall basics of Clifford Algebras and establish notations.
For an elementary introduction to Clifford algebras see, for example, \cite{G}.

Let $V$ be an $n$-dimensional real vector space,
and $Q$ a quadratic form on $V$.
The form $Q$ uniquely extends to a bilinear form on $V$.
Diagonalizing $Q$, we have that there exist integers $p$, $q$ and orthogonal
basis $\{e_1,\dots, e_n\}$ of $(V,Q)$, such that
\begin{equation*}
Q(e_j)=
\begin{cases} 1 & 1 \leq j \leq p;\\
-1 & p+1 \leq j \leq p+q; \\
0 & p+q < j \leq n. \end{cases}
\end{equation*}
By Sylvester's Law of Inertia, we the numbers $p$ and $q$ are independent of
basis chosen. The ordered pair $(p,q)$ is the signature of $(V,Q)$.
From this point on we restrict our attention to non-degenerate quadratic forms
$Q$, in which case $p+q=n$.
If $Q$ is such a form with $q=0$ or $p=0$, then it is positive or negative
definite quadratic form respectively.

Recall the standard construction of the universal Clifford algebra associated
to $(V,Q)$ as a quotient of the tensor algebra.
We start with the tensor algebra over $V$,
\begin{equation*}
\bigotimes V=\R\oplus V\oplus (V\otimes V)\oplus (V \otimes V \otimes V)
\oplus \dots,
\end{equation*}
consider a set of elements of the tensor algebra
\begin{equation*}
S=\{v\otimes v+Q(v): v\in V\} \quad \subset \bigotimes V,
\end{equation*}
and let $(S)$ denote the ideal of $\bigotimes V$ generated by the elements
of $S$. Then the universal Clifford algebra associated to $(V,Q)$ can be
defined as a quotient
\begin{equation}  \label{quotientdef}
{\mathcal A}_{Q}=\bigotimes V/(S).
\end{equation}
We note the sign convention chosen for elements of the ideal is not standard
in the literature, resulting in a possible interchange of $p$ and $q$.
Let $e_0 \in {\mathcal A}_Q$ denote the multiplicative identity of the algebra.
If $v_1$ and $v_2$ are orthogonal elements of $(V,Q)$,
\begin{multline}  \label{anticommuteeq}
v_1 \otimes v_2 + v_2 \otimes v_1
= (v_1+v_2)\otimes(v_1+v_2) - v_1\otimes v_1 - v_2\otimes v_2  \\
= -Q(v_1+v_2)+Q(v_1)+Q(v_2)=0.
\end{multline}
Thus ${\mathcal A}_Q$ is a finite-dimensional algebra over $\R$ generated by
$e_1,\dots,e_n$, and let ${\mathcal A}_{p,q}$ denote the algebra
${\mathcal A}_Q$ constructed from $(V,Q)$, where $Q$ has signature $(p,q)$.
We fix an orthogonal basis
\begin{equation*}
  \{e_1,e_2,\dots,e_p,\tilde{e}_{p+1},\dots,\tilde{e}_{p+q}\}
\end{equation*}
of $V$ such that $Q(e_j)=1$ and $Q(\tilde{e}_j)=-1$ for all applicable $j$.
In the positive definite case the basis of $V$ is given by
$\{e_1,e_2,\dots,e_n\}$, and we let ${\mathcal A}_n={\mathcal A}_{n,0}$.
We have the following relations for $1\leq i,j\leq n$ with $i\neq j$,
\begin{align}
&e_0^2=e_0, \qquad e_i^2=-e_0, \qquad \tilde{e}_j^2=e_0, \qquad
e_0e_i=e_ie_0=e_i, \qquad  e_0\tilde{e_i}=\tilde{e_i}e_0=\tilde{e_i},
\label{basicrelations1} \\
&e_ie_j=-e_j e_i, \qquad e_i\tilde{e_j}=-\tilde{e_j} e_i, \qquad
\tilde{e_i}\tilde{e_j}=-\tilde{e_j}\tilde{e_i}.  \label{basicrelations2}
\end{align}
We consider subsets $B \subseteq \{1,\dots, n\}$.
If $B$ has $k>0$ elements,
\begin{equation*}
B=\{i_1,i_2,\dots,i_k\}\subseteq \{1,\dots, n\}
\end{equation*}
with $i_1<i_1< \dots <i_k$, define
\begin{equation} \label{basisdef}
e_B=e_{i_1i_2\dots i_k}=e_{i_1}\otimes e_{i_2} \otimes \dots \otimes \tilde{e}_{i_k}
\end{equation}
or, more precisely, $e_B$ is the image of this tensor product in
${\mathcal A}_Q$.
If $B$ is empty, we set $e_{\emptyset}$ be the identity element $e_0$.
These elements
\begin{equation*}
\bigl\{e_B: B \subseteq \{1,2,\dots, n\} \bigr\}
\end{equation*}
form a vector space basis of ${\mathcal A}_{p,q}$ over $\R$.
We will be especially concerned with the space $\R\oplus V$,
which we identify with the vector subspace $\R^{p+q+1}\subset {\mathcal A}_{p,q}$
spanned by
\begin{equation*}
  \{e_0,e_1,e_2,\dots,e_p,\tilde{e}_{p+1},\dots,\tilde{e}_{p+q}\}.
\end{equation*}
Any $X \in \R^{p+q+1}$ can be written as
\begin{equation*}
X=\sum_{j=0}^{p} x_j e_j+\sum_{j=p+1}^{p+q} \tilde{x}_j \tilde{e}_j.
\end{equation*}
In particular, in the positive definite case any
$X \in \R^{n+1}\subset {\mathcal A}_n$ can be expressed as
$X=\sum_{j=0}^{n} x_j e_j$.

We can also construct a complex universal Clifford algebra from a quadratic
space over $\C$.
Let $V^\C$ be an $n$-dimensional complex vector space with quadratic form $Q$.
Diagonalizing $Q$, we have that there exists an orthogonal basis
$\{e_1,\dots, e_n\}$ of $V^\C$ and integer $p$ such that $Q(e_j)=1$ if
$1\leq j\leq p$ and $Q(e_j)=0$ otherwise.
The form $Q$ is non-degenerate if and only if $p=n$, and we only consider
this case.
We perform a similar construction to \eqref{quotientdef}:
form the tensor algebra over $V^\C$
\begin{equation*}
\bigotimes V^\C = \C\oplus V^\C\oplus (V^\C\otimes V^\C) \oplus
(V^\C\otimes V^\C \otimes V^\C) \oplus \dots,
\end{equation*}
consider a set
\begin{equation*}
S^\C = \{v\otimes v+Q(v): v\in V^\C\} \quad \subset \bigotimes V^\C,
\end{equation*}
and let $(S^\C)$ denote the ideal generated by $S^\C$.
Then the complex universal Clifford algebra associated to $(V^\C,Q)$ is a
quotient
\begin{equation*}
{\mathcal A}_{Q}^\C=\bigotimes V^\C /(S^\C).
\end{equation*}
We also use notation ${\mathcal A}_n^\C$ for this algebra.
The analogues of \eqref{anticommuteeq} and
\eqref{basicrelations1}-\eqref{basicrelations2} remain valid for
${\mathcal A}_n^\C$, and the products $e_B$ defined as in \eqref{basisdef}
form a basis of ${\mathcal A}_n^\C$.

We can also consider ${\mathcal A}_{n}^\C$ as a real algebra generated by
\begin{equation*}
  e_0,e_1,e_2,\dots,e_n,ie_0,ie_1,ie_2,\dots,ie_n.
\end{equation*}
By the universality property of ${\mathcal A}_{p,q}$, the natural inclusion
$\iota: V \hookrightarrow V^\C$ defined on basis vectors by 
\begin{equation}  \label{gen-rel}
\iota(e_j)=e_j, \quad 1 \le j \le p, \qquad
\iota(\tilde{e}_j)=ie_j, \quad p+1 \leq j \leq p+q,
\end{equation}
extends to an injective $\R$-algebra homomorphism
$\iota: {\mathcal A}_{p,q} \hookrightarrow {\mathcal A}_{p+q}^\C$.
Thus, we can consider ${\mathcal A}_{p,q}$ as a unital real subalgebra of
${\mathcal A}_{p+q}^\C$.

We identify the $\C$-span of $e_0, e_1,\dots,e_n$ with
$\C^{n+1} \subset {\mathcal A}_n^\C$.
On $\C^{n+1}$ we have the {\em Clifford conjugation} defined by
\begin{equation*}
Z=z_0e_0+\sum_{j=1}^n z_je_j \quad \mapsto \quad Z^+=z_0e_0-\sum_{j=1}^n z_je_j
\end{equation*}
and the {\em complex conjugation} defined by
\begin{equation*}
Z = z_0e_0+\sum_{j=1}^n z_je_j \quad \mapsto \quad
\bar{Z}=\bar{z_0}e_0+\sum_{j=1}^n \bar{z_j}e_j.
\end{equation*}
(These conjugations can be extended to all of ${\mathcal A}_{p+q}^\C$,
but for the purposes of this paper we do not need this.)
The Clifford conjugation fixes the $\C$-span of $e_0$ identified with $\C$,
and the complex conjugation fixes the $\R$-span of
$e_{0},\dots,e_n$ identified with $\R^{n+1}$, and these conjugations
can be viewed as reflecting over the respective subspaces.

These two conjugations commute and lead to two useful quadratic forms on
$\C^{n+1} \subset {\mathcal A}_n^\C$. The first such form is
\begin{equation*}
N(Z)=ZZ^+=Z^+Z=\sum_{j=0}^n z_j^2
\end{equation*}
(note that it is complex valued). The corresponding bilinear form is
\begin{equation}  \label{bilinear-form}
\langle Z,W\rangle=\frac12 (Z^+W+ZW^+) = \sum_{j=0}^{n} z_jw_j.
\end{equation}
We denote by $\mathcal{N}_n^{\C}$ the null cone of $N(Z)$:
\begin{equation*}
\mathcal{N}_n^{\C}=\{ Z \in \C^{n+1} ;\: N(Z) = 0\},
\end{equation*}
then all $Z \in \C^{n+1} \setminus \mathcal{N}_n^{\C}$ are invertible with
inverse given by $Z^{-1} = N(Z)^{-1} Z^+$.
The other quadratic form is real valued:
\begin{equation*}
\norm{Z}^2 = \frac12 (Z\bar{Z}^++\bar{Z}Z^+) = \frac12 (\bar{Z}^+Z+Z^+\bar{Z})
= \sum_{j=0}^n \abs{z_j}^2.
\end{equation*}

As was mentioned earlier, we consider ${\mathcal A}_{p,q}$ as a real subalgebra
of ${\mathcal A}_{p+q}^\C$.
We describe the restrictions of these two conjugations and two quadratic forms
to $\R^{p+q+1} \subset {\mathcal A}_{p,q} \subset {\mathcal A}_{p+q}^\C$:
\begin{equation*}
X=x_0e_0+\sum_{j=1}^p x_j e_j+\sum_{j=p+1}^{p+q} \tilde{x}_j \tilde{e}_j
\quad \mapsto \quad
X^+=x_0e_0-\sum_{j=1}^p x_j e_j-\sum_{j=p+1}^{p+q} \tilde{x}_j \tilde{e}_j,
\end{equation*}
\begin{equation}  \label{compl-conj-Rn}
X=x_0e_0+\sum_{j=1}^p x_j e_j+\sum_{j=p+1}^{p+q} \tilde{x}_j \tilde{e}_j
\quad \mapsto \quad
\bar{X}=x_0e_0+\sum_{j=1}^p x_j e_j-\sum_{j=p+1}^{p+q} \tilde{x}_j \tilde{e}_j,
\end{equation}
\begin{equation*}
N(X)=XX^+=X^+X = \sum_{j=0}^p x_j^2-\sum_{j=p+1}^{p+q}\tilde{x}_j^2,
\end{equation*}
\begin{equation}  \label{Eucl-norm}
\norm{X}^2 = \frac12 (X\bar{X}^++\bar{X}X^+) = \frac12 (\bar{X}^+X+X^+\bar{X})
=\sum_{j=0}^{p} x_j^2+\sum_{j=p+1}^{q} \tilde{x}_j^2.
\end{equation}
The bilinear form on $\R^{p+q+1}$ corresponding to $N(X)$ is
\begin{equation*}
\langle X,Y \rangle =\frac{1}{2}(X^+Y+XY^+)
=\sum_{j=0}^p x_jy_j-\sum_{j=p+1}^{p+q} \tilde{x}_j\tilde{y}_j.
\end{equation*}
In line with the complex case, we consider the null cone
\begin{equation*}
\mathcal{N}_{p,q}= \{X \in \R^{p+q+1};\: N(X)=0 \},
\end{equation*}
then all $X \in \R^{p+q+1} \setminus \mathcal{N}_{p,q}$
are invertible with inverse given by $X^{-1} = N(X)^{-1}X^+$.
Finally, we note that in the positive definite case $\mathcal{N}_{n,0}=\{0\}$.

\section{Dirac Operators, Monogenic Functions and Green's Functions}  \label{3}

In this section we recall the definitions of Dirac operators,
monogenic functions and Green's functions in the setting of complex
Clifford algebras that were originally introduced in \cite{R1}.
Then, using the inclusion
$\R^{p+q+1} \subset {\mathcal A}_{p,q} \subset {\mathcal A}_{p+q}^\C$,
we restrict these notions to define their analogues in the setting of real
Clifford algebras ${\mathcal A}_{p,q}$.
When we apply these restricted Dirac operators to the restricted
Green's functions, we obtain special monogenic functions that will
serve as reproducing kernels of Cauchy's integral formulas in the
${\mathcal A}_{p,q}$ setting.

We introduce linear differential operators on $\C^{n+1}$
\begin{equation*}
\nabla^+_{\C} =e_0\pdv{z_0} + \sum_{j=1}^n e_j\pdv{z_j}
\qquad \text{and} \qquad
\nabla_{\C} = e_0\pdv{z_0} - \sum_{j=1}^n e_j\pdv{z_j}
\end{equation*}
which may be applied to functions on the left and on the right. 
Let $\Box_\C$ be the complex Laplacian
\begin{equation*}
\Box_\C = \sum_{j=0}^n \pdv[2]{z_j}.
\end{equation*}
Then
\begin{equation}  \label{Laplace-factorization}
\nabla_\C \nabla^+_\C f=\nabla^+_\C \nabla_\C  f=\Box_\C f
\qquad \text{and} \qquad
g\nabla_\C\nabla^+_\C=g\nabla^+_\C \nabla_\C =\Box_\C g,
\end{equation}
where $f$ is a holomorphic function on $\C^{n+1}$ with values in
${\mathcal A}_n^\C$ or a left ${\mathcal A}_n^\C$-module, and
$g$ is a holomorphic function on $\C^{n+1}$ with values in
${\mathcal A}_n^\C$ or a right ${\mathcal A}_n^\C$-module.

\begin{definition}
Let $U \subseteq \C^{n+1}$ be an open set, and $M_n^\C$ a left
${\mathcal A}_n^\C$-module. A holomorphic function $f: U \to M_n^\C$
is called {\em complex left-monogenic} if
\begin{equation*}
\nabla^+_{\C} f=e_0\pdv{f}{z_0}+\sum_{j=1}^n e_j\pdv{f}{z_j} =0
\end{equation*}
at all points in $U$.

Similarly, let $\tilde{M}_n^\C$ be a right ${\mathcal A}_n^\C$-module,
then a holomorphic function $g: U\to \tilde{M}_n^\C$
is called {\em complex right-monogenic} if
\begin{equation*}
g\nabla^+_\C=\pdv{g}{z_0}e_0+\sum_{j=1}^n \pdv{g}{z_j}e_j =0
\end{equation*}
at all points in $U$.
\end{definition}

The factorization \eqref{Laplace-factorization} leads to a natural method of
constructing complex left- and right-monogenic functions.
If $\phi: U \to M_n^\C$ is complex harmonic (i.e., $\Box_\C \phi =0$),
with $U$ and $M_n^\C$ as in the definition, then $\nabla_\C f$ is
complex left-monogenic.
Similarly, if $\tilde{\phi}: U\to \tilde{M}_n^\C$ is complex harmonic,
$\tilde{\phi} \nabla_\C$ is complex right-monogenic.

We restrict these definitions to the subalgebra
${\mathcal A}_{p,q} \subset {\mathcal A}_{p+q}^\C$.
Thus we introduce linear differential operators
\begin{align*}
\nabla^+_{p,q} &= e_0\pdv{x_0} + \sum_{j=1}^p e_j\pdv{x_j}
- \sum_{j=p+1}^{p+q}\tilde{e}_j \pdv{\tilde{x}_j}
\qquad \text{and} \\
\nabla_{p,q} &= e_0\pdv{x_0} - \sum_{j=1}^p e_j\pdv{x_j}
+ \sum_{j=p+1}^{p+q} \tilde{e}_j\pdv{\tilde{x}_j} 
\end{align*}
which may be applied to functions on the left and on the right.
Let $\Box_{p,q}$ be a wave operator on $\R^{p+q+1}$,
\begin{equation*}
\Box_{p,q} = \sum_{j=0}^p \pdv[2]{x_j}-\sum_{j=p+1}^{p+q}\pdv[2]{\tilde{x}_j}.
\end{equation*}
Then
\begin{equation}  \label{wave-factorization}
\nabla_{p,q} \nabla^+_{p,q} f=\nabla^+_{p,q} \nabla_{p,q}  f=\Box_{p,q} f
\qquad \text{and} \qquad
g\nabla_{p,q} \nabla^+_{p,q}=g\nabla^+_{p,q} \nabla_{p,q}=\Box_{p,q} g,
\end{equation}
where $f$ is a ${\mathcal C}^2$ function on $\R^{p+q+1}$ with values in
${\mathcal A}_{p,q}$ or a left ${\mathcal A}_{p,q}$-module, and
$g$ is a ${\mathcal C}^2$ function on $\R^{p+q+1}$ with values in
${\mathcal A}_{p,q}$ or a right ${\mathcal A}_{p,q}$-module.

\begin{definition}
Let $U \subset \R^{p+q+1}$ be an open set, and $M_{p,q}$ a left
${\mathcal A}_{p,q}$-module.
A ${\mathcal C}^1$ function $f: U \to M_{p,q}$ is called
{\em $(p,q)$-left-monogenic} if
\begin{equation*}
\nabla^+_{p,q} f = e_0\pdv{f}{x_0}+\sum_{j=1}^p e_j\pdv{f}{x_j}
- \sum_{j=p+1}^{p+q}\tilde{e}_j \pdv{f}{\tilde{x}_j} =0
\end{equation*}
at all points in $U$.

Similarly, let $\tilde{M}_{p,q}$ be a right ${\mathcal A}_{p,q}$-module, then
a ${\mathcal C}^1$ function $g: U \to \tilde{M}_{p,q}$ is called
{\em $(p,q)$-right-monogenic} if
\begin{equation*}
g \nabla^+_{p,q} = \pdv{g}{x_0}e_0 + \sum_{j=1}^p \pdv{g}{x_j}e_j
- \sum_{j=p+1}^{p+q}\pdv{g}{\tilde{x}_j}\tilde{e}_j =0
\end{equation*}
at all points in $U$.
\end{definition}

The factorization \eqref{wave-factorization} leads to a natural method of
constructing $(p,q)$-left and right-monogenic functions.
If $\phi: U \to M_{p,q}$ satisfies $\Box_{p,q} \phi =0$,
with $U$ and $M_{p,q}$ as in the definition, then $\nabla_{p,q} f$ is
$(p,q)$-left-monogenic.
Similarly, if $\tilde{\phi}: U \to \tilde{M}_{p,q}$ satisfies
$\Box_{p,q} \tilde{\phi} =0$, 
$\tilde{\phi} \nabla_{p,q}$ is $(p,q)$-right-monogenic.

The following case is an important example of such a construction,
in which we consider a Green's function on
$\C^{p+q+1} \subset {\mathcal A}_{p+q}^C$
and a suitable restriction to $\R^{p+q+1} \subset {\mathcal A}_{p,q}$.

\begin{definition}  \label{defthatprecedessquareroot}
For odd $n \geq 2$, we define the following function
$H_n(Z)$: $\C^{n+1}\setminus \mathcal{N}_n^\C \to \C$ by
\begin{equation*}
H_n(Z) = \frac1{N(Z)^{\frac{n-1}2}} = \frac{1}{(\sum_{j=0}^n z_j^2)^{\frac{n-1}2}}.
\end{equation*}
\end{definition}

We note that if $n$ is even, this may not be well defined.
In this case, we introduce a domain
\begin{equation*}
\C_G^{n+1}=\C^{n+1}\setminus \{ Z \in \C^{n+1} ;\: N(Z) \in \R, N(Z)\leq 0\}
\end{equation*}
and choose a branch of $N(Z)^{\frac{1}{2}}: \C_G^{n+1} \to \C$ with values
in the right half-plane $\{ w \in \C ;\: \Re w> 0 \}$.
Then $H_{n}(Z) = \bigl( N(Z)^{\frac{1}{2}} \bigr)^{1-n}$ is a well defined
function $\C_{G}^{n+1}\to \C$.
From now on we restrict the domain of $H_n(Z)$ to $\C_G^{n+1}$
regardless of the parity of $n$, as the same proofs hold in either case.
Observe that on this domain $H_n(Z)$ is complex harmonic,
and, since $H_n(Z)$ is scalar-valued, $\nabla_\C H_n(Z) = H_n(Z)\nabla_\C$
is complex left- and right-monogenic.
Thus, we obtain a function on $\C_G^{n+1}$ with values in
${\mathcal A}_n^\C$ that is complex left- and right-monogenic:

\begin{lemma}
An ${\mathcal A}_n^\C$-valued function $G_n(Z)$ defined on $\C_G^{n+1}$ as
\begin{equation*}
G_n(Z)= \frac1{1-n} \nabla_\C H_n(Z)
= \frac{z_0-\sum_{j=1}^n z_j e_j}{(\sum_{j=0}^n z_j^2)^{\frac{n+1}2}}
= \frac{Z^+}{N(Z)^{\frac{n+1}2}}
\end{equation*}
is complex left- and right-monogenic.
\end{lemma}

Restricting this function to $\R^{p+q+1} \subset {\mathcal A}_{p,q}$
as was done with Dirac operators earlier, we obtain the following solution to
the wave equation $\Box_{p,q} \phi =0$ and the corresponding $(p,q)$-left-
and $(p,q)$-right-monogenic functions.
Let
\begin{equation*}
\R_G^{p+q+1} =\R^{p+q+1} \setminus \{ X \in \R^{n+1} ;\: N(X)\leq 0\}
= \R^{p+q+1} \cap \C_G^{p+q+1}.
\end{equation*}

\begin{definition}
For all $p+q \geq 2$, define the functions $H_{p,q}(X): \R_G^{p+q+1} \to \R$
and $G_{p,q}(X): \R_G^{p+q+1} \to {\mathcal A}_{p,q}$ by
\begin{equation*}
H_{p,q}(X) = H_{p+q}(X) \eval_{\R_G^{p+q+1}} = \frac1{N(X)^{\frac{p+q-1}2}}
= \frac1{\bigl( \sum_{j=0}^p x_j^2-\sum_{j=p}^{p+q} \tilde{x}_j^2
  \bigr)^{\frac{p+q-1}2}},
\end{equation*}
\begin{multline*}
  G_{p,q}(X) = G_{p+q}(X) \eval_{\R_G^{p+q+1}}
= \frac1{1-p-q} \nabla_{p,q} H_{p,q}(X) \\
= \frac{x_0-\sum_{j=1}^p x_j e_j - \sum_{j=p+1}^{p+q} \tilde{x}_j \tilde{e}_j}
{\bigl( \sum_{j=0}^p x_j^2-\sum_{j=p+1}^{p+q} \tilde{x}_j^2 \bigr)^{\frac{p+q+1}2}}
= \frac{X^+}{N(X)^{\frac{p+q+1}2}}.
\end{multline*}
\end{definition}

Consequently, $\Box_{p,q} H_{p,q}(X)= 0$ and $G_{p,q}$ is
$(p,q)$-left- and right-monogenic.


\section{Differential Forms}  \label{4}

In this section we introduce the Clifford-valued differential $n$-forms
$D_nz$ and $D_{p,q}x$ that will appear in the statement of Cauchy's
integral formulas.
We also prove several properties of these forms.
At the end of the section we state the classical Cauchy's integral formulas
for monogenic functions.
The complex case originally was developed in \cite{R1}.

\begin{definition}  \label{dvcomplexdef}
Let $dV_{\C}$ be the $n+1$ complex holomorphic form on $\C^{n+1}$:
\begin{equation}
dV_{\C} = dz_0 \wedge dz_1\wedge \dots \wedge dz_n;
\end{equation}
it is normalized so that $dV_\C(e_0, e_1, e_2,\dots, e_n)=1$.
\end{definition}

Recall the bilinear form \eqref{bilinear-form} on $\C^{n+1}$.

\begin{definition}
Let $D_nz$ be the unique $\C^{n+1}$-valued complex holomorphic $n$-form
on $\C^{n+1}$ such that, for all $Z_0,Z_1,\dots,Z_n \in \C^{n+1}$, we have
\begin{equation}  \label{innerproductform}
\bigl\langle Z_0,D_nz(Z_1,Z_2,\dots,Z_n) \bigr\rangle
= dV_\C(Z_0, Z_1, Z_2,\dots Z_n).
\end{equation}
\end{definition}

Explicitly, we can express $D_n{z}$ as a sum of $n+1$ terms by substituting
basis vectors into \eqref{innerproductform} yielding
\begin{equation}  \label{Dzsumdef}
D_nz=\sum_{j=0}^n (-1)^{j}e_j d\hat{z}_{j},
\end{equation}
where
\begin{equation*}
d\hat{z}_{j} =
dz_0\wedge dz_1\wedge \dots \wedge dz_{j-1} \wedge dz_{j+1} \wedge \dots dz_n.
\end{equation*}

\begin{lemma}  \label{productrulez}
Let $U \subset \C^{n+1}$ be an open set and let $f: U \to {\mathcal A}_n^\C$, 
$g: U \to \tilde{M}_n^\C$ be holomorphic functions, where $\tilde{M}_n^\C$ is a
right ${\mathcal A}_n^\C$-module. Then,
\begin{equation}  \label{productrulez1}
d(gD_nzf)=(g\nabla^+_\C)fdV_\C+g(\nabla^+_\C f) dV_\C
\qquad \text{and} \qquad
d(gD_nz)=(g\nabla^+_\C)dV_\C.
\end{equation}
Similarly, if $f: U \to M_n^\C$ and $g: U \to {\mathcal A}_n^\C$
are holomorphic functions, where $M_n^\C$ is a left ${\mathcal A}_n^\C$-module,
\begin{equation*}
d(gD_nzf)=(g\nabla^+_\C)fdV_\C+g(\nabla^+_\C f) dV_\C
\qquad \text{and} \qquad
d(D_nzf)= (\nabla^+_\C f) dV_\C.
\end{equation*}
\end{lemma}



\begin{corollary}  \label{stokesapp}
A holomorphic function $f$ on $U$ is complex left-monogenic if and only if
the $n$-form $D_nzf$ is closed on $U$.
Similarly, a holomorphic function $g$ on $U$ is complex right-monogenic
if and only if the $n$-form $gD_nz$ is closed on $U$.
\end{corollary}

This corollary may be used to give an alternative definition of complex
left- and right-monogenic functions.

Recall that the generators of the real universal Clifford algebra
${\mathcal A}_{p,q}$ and complex universal Clifford algebra
${\mathcal A}_{p+q}^{\C}$ are related by \eqref{gen-rel}.
Therefore, restricting to $\R^{p+q+1}$, we have
\begin{align*}
& dz_j(x_j e_j)=x_j=dx_j(x_je_j),
\quad \text{hence} \quad
dz_j \eval_{\R^{p+q+1}}=dx_j \qquad\text{if } 0 \leq j \leq p;  \\
& dz_j(\tilde{x}_j\tilde{e}_j) = i\tilde{x}_j
= id\tilde{x}_j(\tilde{x}_j\tilde{e}_j),
\quad \text{hence} \quad
dz_j \eval_{\R^{p+q+1}} = id\tilde{x}_j  \qquad\text{if } p+1 \leq j \leq p+q.
\end{align*}

\begin{definition}
We define $dV_{p,q}$ to be the restriction of $dV_{\C}$ to $\R^{p+q+1}$:
\begin{equation}  \label{factorsofi}
dV_{p,q} = dV_{\C} \eval_{\R^{p+q+1}}  =
i^q dx_0 \wedge \dots \wedge dx_p \wedge d\tilde{x}_{p+1} \wedge \dots
\wedge d\tilde{x}_{p+q};
\end{equation}
it is normalized so that
$dV_{p,q}(e_0,e_1,\dots,e_{p},\tilde{e}_{p+1},\dots,\tilde{e}_{p+q})=i^{q}$.
\end{definition}

\begin{definition}  \label{Dxinnerproductdef}
Let $D_{p,q}x$ be the restriction of $D_{p+q}z$ to $\R^{p+q+1}$,
so that we have, for all $(X_0,X_1,\dots,X_{p+q}) \in \R^{p+q+1}$,
\begin{equation*}
\bigl\langle X_0, D_{p,q}x(X_1,X_2,\dots,X_{p+q}) \bigr\rangle
= dV_{p,q}(X_0, X_1, \dots X_{p+q}).
\end{equation*}
\end{definition}

Explicitly,
\begin{equation}  \label{Dx-explicit}
D_{p,q}x= i^q \biggl( \sum_{j=0}^{p} (-1)^{j}e_j d\hat{x}_{j}
- \sum_{j=p+1}^{p+q} (-1)^{j}\tilde{e}_j d\hat{\tilde{x}}_{j} \biggr),
\end{equation}
where
\begin{align*}
d\hat{x}_j &= dx_0\wedge \dots \wedge dx_{j-1} \wedge dx_{j+1} \wedge \dots
\wedge dx_{p}\wedge d\tilde{x}_{p+1}\wedge \dots \wedge d\tilde{x}_{p+q}
\qquad\text{if } 0 \leq j \leq p,  \\
d\hat{\tilde{x}}_j &= dx_0\wedge\dots\wedge dx_{p}\wedge d\tilde{x}_{p+1}\wedge
\dots \wedge d\tilde{x}_{j-1} \wedge d\tilde{x}_{j+1} \wedge \dots
\wedge d\tilde{x}_{p+q}
\qquad\text{if } p+1 \leq j \leq p+q.
\end{align*}

We have the following real analogue of Lemma \ref{productrulez};
its proof is the same as in the complex case.

\begin{lemma}  \label{productrulex}
Let $U \subset \R^{p+q+1}$ be an open set and let $f: U \to {\mathcal A}_{p,q}$, 
$g: U \to \tilde{M}_{p,q}$ be ${\mathcal C}^1$ functions, where
$\tilde{M}_{p,q}$ is a right ${\mathcal A}_{p,q}$-module. Then,
\begin{equation*}
d(gD_{p,q}xf)=(g\nabla^+_{p,q})f dV_{p,q} + g(\nabla^+_{p,q}f) dV_{p,q}
\qquad \text{and} \qquad
d(gD_{p,q}x)=(g\nabla^+_{p,q}) dV_{p,q}.
\end{equation*}
Similarly, if $f: U \to M_{p,q}$ and $g: U \to {\mathcal A}_{p,q}$
are ${\mathcal C}^1$ functions, where $M_{p,q}$ is a left
${\mathcal A}_{p,q}$-module,
\begin{equation*}
d(gD_{p,q}xf)=(g\nabla^+_{p,q})f dV_{p,q} + g(\nabla^+_{p,q}f) dV_{p,q}
\qquad \text{and} \qquad
d(D_{p,q}xf) = (\nabla^+_{p,q}f) dV_{p,q}.
\end{equation*}
\end{lemma}

\begin{corollary}  \label{closedcorrolary}
A ${\mathcal C}^1$ function $f$ on $U$ is $(p,q)$-left-monogenic if and only if
the $n$-form $D_{p,q}xf$ is closed on $U$.
Similarly, a ${\mathcal C}^1$ function $g$ on $U$ is $(p,q)$-right-monogenic
if and only if the $n$-form $gD_{p,q} x$ is closed on $U$.
\end{corollary}

Again, this corollary may serve as an alternative definition of
$(p,q)$-left- and $(p,q)$-right-monogenic functions.

We orient $\R^{p+q+1} \subset {\mathcal A}_{p,q}$ so that
\begin{equation*}
  \{e_0,e_1,e_2,\dots,e_p,\tilde{e}_{p+1},\dots,\tilde{e}_{p+q}\}
\end{equation*}
is a positively oriented basis.
For an open subset $U \subset \R^{p+q+1}$ with piecewise smooth boundary
$\partial U$, at each smooth boundary point $X \in \partial U$, let
\begin{equation*}
{\bf n}_X = n_0e_0+n_1e_1+\dots+n_pe_p+\tilde{n}_{p+1}\tilde{e}_{p+1}+\dots
+\tilde{n}_{p+q}\tilde{e}_{p+q}
\end{equation*}
be an outward pointing normal unit vector to $U$ at $X$
(with respect to \eqref{Eucl-norm}).
Then we can orient the boundary $\partial U$ so that a basis
$\{v_1,\dots,v_{p+q}\}$ of the tangent space $T_XU$ is positively oriented
if and only if $\{{\bf n}_X,v_1,\dots,v_{p+q}\}$ is a positively oriented
basis of $\R^{p+q+1}$.
Let $dS_{p,q}$ be the contraction of ${\bf n}_X$ with $dV_{p,q}$, so that
\begin{equation*}
dS_{p,q}(v_1,\dots,v_{p+q}) = dV_{p,q}({\bf n}_X,v_1,\dots,v_{p+q})
\end{equation*}
whenever $v_1,\dots,v_{p+q}$ are vectors in the tangent space $T_XU$ of
$\partial U$ at $X$.
Explicitly,
\begin{equation}  \label{dSexplicit}
dS_{p,q} = i^{q} \biggl( \sum_{j=0}^{p} (-1)^j n_jd\hat{x}_j
+\sum_{j=p+1}^{p+q}(-1)^jn_jd\hat{\tilde{x}}_j \biggr).
\end{equation}

\begin{lemma}
With $\bar{\bf n}_X$ denoting the complex conjugate of ${\bf n}_X$
(recall \eqref{compl-conj-Rn}), the restriction
\begin{equation*}
D_{p,q}x \eval_{\partial U} = \bar{\bf n}_X dS_{p,q}.
\end{equation*}
\end{lemma}

\begin{proof}
By definition, ${\bf n}_X$ is orthogonal to the tangent space $T_X \partial U$.
Thus, for every vector $X=x_0e_0+x_1e_1+\dots+x_pe_p
+\tilde{x}_{p+1}\tilde{e}_{p+1}+\dots+\tilde{x}_{p+q}\tilde{e}_{p+q}$
in $T_X \partial U$, we have
\begin{equation*}
\sum_{j=0}^pn_jx_j + \sum_{j=p+1}^{p+q} \tilde{n}_j \tilde{x}_{j}=0,
\end{equation*}
and hence
\begin{equation*}
n_0dx_0+\dots+n_{p}dx_p+\tilde{n}_{p+1}d\tilde{x}_{p+1}+\dots
+ \tilde{n}_{p+q}d\tilde{x}_{p+q}=0.
\end{equation*}
Without loss of generality we may suppose $n_0 \neq 0$,
then we can isolate $dx_0$:
\begin{equation}  \label{isodx0}
dx_0=-\frac{1}{n_0}(n_1dx_1+ n_2dx_2+\dots + n_pdx_{p}
+ \tilde{n}_{p+1}d\tilde{x}_{p+1}+\dots +\tilde{n}_{p+q}d\tilde{x}_{p+q}).
\end{equation}
We introduce the symbol $d\check{x}_j$ denoting wedge products of all but
$dx_0$, $dx_j$:
\begin{align*}
d\check{x}_j &= dx_1\wedge \dots \wedge dx_{j-1} \wedge dx_{j+1} \wedge \dots
\wedge dx_{p}\wedge d\tilde{x}_{p+1}\wedge \dots \wedge d\tilde{x}_{p+q}
\qquad\text{if } 1 \leq j \leq p,  \\
d\check{\tilde{x}}_j &= dx_1\wedge\dots\wedge dx_{p}\wedge d\tilde{x}_{p+1}\wedge
\dots \wedge d\tilde{x}_{j-1} \wedge d\tilde{x}_{j+1} \wedge \dots
\wedge d\tilde{x}_{p+q}
\qquad\text{if } p+1 \leq j \leq p+q.
\end{align*}
Then
\begin{align*}
dx_j\wedge d\check{x}_j &=(-1)^{j-1} d\hat{x}_0 \quad\text{if } 1 \leq j \leq p,
\\
d\tilde{x}_j\wedge d\check{\tilde{x}}_j &=(-1)^{j-1} d\hat{x}_0
\quad\text{if } p+1 \leq j \leq p+q.
\end{align*}
Substituting \eqref{isodx0} into the expression \eqref{dSexplicit} for
$dS_{p,q}$ and using these notations,
\begin{multline*}
dS_{p,q} = i^q \biggl( n_0d\hat{x}_0
+ \sum_{j=1}^p (-1)^j \Bigl(-\frac{n_j^2}{n_0}\Bigr) dx_j \wedge d\check{x}_j)
+ \sum_{j=p+1}^{p+q} (-1)^j \Bigl(-\frac{\tilde{n}_j^2}{n_0}\Bigr)
d\tilde{x}_j \wedge d\check{x}_j \biggr) \\
=\frac{i^q}{n_0} \Biggl( \sum_{j=0}^p n_j^2
+ \sum_{j=p+1}^{p+q} \tilde{n}_j^2 \biggr) d\hat{x}_0
=\frac{i^q}{n_0} d\hat{x}_0,
\end{multline*}
since ${\bf n}_X$ is a unit vector.
We perform a similar computation for $D_{p,q}x$ substituting \eqref{isodx0}
into \eqref{Dx-explicit} to get
\begin{multline*}
D_{p,q}x \eval_{\partial U} = i^q \biggl( \sum_{j=0}^p (-1)^j e_j d\hat{x}_j
- \sum_{j=p+1}^{p+q} (-1)^j \tilde{e}_j d\hat{\tilde{x}}_j \biggr)   \\
=i^q \biggl( n_0d\hat{x}_0
+ \sum_{j=1}^p(-1)^j \Bigl(-\frac{n_j}{n_0}\Bigr) e_j dx_j\wedge d\check{x}_j
-\sum_{j=p+1}^{p+q} (-1)^{j} \Bigl(-\frac{\tilde{n}_j}{n_0}\Bigr)
\tilde{e}_j d\tilde{x}_j\wedge d\check{\tilde{x}}_j \biggr)  \\
= \frac{i^q}{n_0} \biggl(\sum_{j=0}^p n_je_j
- \sum_{j=p+1}^{p+q} \tilde{n}_j\tilde{e}_j \biggr) d\hat{x}_0
= \frac{i^q}{n_0} \bar{\bf n}_X d\hat{x}_0
= \bar{\bf n}_X dS_{p,q}.
\end{multline*}
\end{proof}

Let $K_r$ and $S_r$ be the boundaries of the sets
$\{X \in \R^{p+q+1} ;\: N(X)\leq r\}$ and
$\{X \in \R^{p+q+1} ;\: \norm{X}^2 \leq r\}$ respectively.
Note that the outward pointing normal vectors at $X$ will be
respectively $\bar{X} / \norm{X}$ and $X / \norm{X}$,
so we get the following analogue of Lemma 3 of \cite{L}:

\begin{corollary} \label{dSlemma}
\begin{equation*}
D_{p,q}x \eval_{K_r} = \frac{X}{\norm{X}}dS_{p,q}, \qquad
D_{p,q}x \eval_{S_r} = \frac{\bar{X}}{\norm{X}}dS_{p,q}
= \frac{\bar{X}}r dS_{p,q}.
\end{equation*}
\end{corollary}

These two expressions are the same in the positive definite case (when $q=0$).

We state the classical Cauchy's integral formulas for left- and
right-monogenic functions in the positive definite case ($q=0$).
For details and proofs see, for example, \cite{BDS, DSS, GM}.
Recall that in this case we write ${\mathcal A}_{n}$ instead of
${\mathcal A}_{n,0}$.

\begin{theorem}  \label{standardintegralformula}
Let $n \geq 2$.
Let $U \subset \R^{n+1}\subset {\mathcal A}_{n}$ be an open bounded set with
piecewise ${\mathcal C}^1$ boundary $\partial U$, and let $M_n$ be a left
${\mathcal A}_{n}$-module.
Suppose that an $M_n$-valued function $f$ is $(n,0)$-left-monogenic function
on a neighborhood of the closure $\overline{U}$. We have:
\begin{equation*}
\int_{\partial U} G_{n,0}(X-X_0)D_{n,0}x f(X) =
\begin{cases} \omega_n f(X_0) & \text{if } X_0\in U, \\
0 & \text{if } X_0 \notin \overline{U}, \end{cases}
\end{equation*}
where $\omega_n$ denotes the $n$-dimensional volume of the unit $n$-sphere
in $\R^{n+1}$.

Similarly, if $\tilde M_n$ is a right ${\mathcal A}_{n}$-module and $g$ is an
$\tilde M_n$-valued function that is $(n,0)$-right-monogenic on a neighborhood
of $\overline{U}$, then
\begin{equation*}
\int_{\partial U} g(X)D_{n,0}x G_{n,0}(X-X_0) =
\begin{cases} \omega_n g(X_0) & \text{if } X_0\in U, \\
0 & \text{if } X_0\notin \overline{U}. \end{cases}
\end{equation*}
\end{theorem}

\section{First Cauchy's Integral Formula}  \label{5}

In this section we present an integral formula for
$(p,q)$-monogenic functions that have complex holomorphic extension to
some open neighborhood in $\C^{p+q+1}$.
The statement of the formula involves a one-parameter deformation map
$h_{\epsilon}$ that is used to deform the contour of integration so that
it does not cross the singularities of $N(X-X_0)^{-\frac{p+q+1}2}$.

\begin{definition}
For $0\leq \epsilon\leq 1$, let $h_{\epsilon}: \C^{p+q+1} \to \C^{p+q+1}$
be a linear map defined by
\begin{equation*}
Z=\sum_{j=0}^{p+q}z_j e_j \quad \mapsto \quad
h_{\epsilon}(Z)=\sum_{j=0}^p (1+i\epsilon)z_je_j
+\sum_{j=p+1}^{q+1} (1-i\epsilon)z_je_j.
\end{equation*}
\end{definition}

Note that the map $h_{\epsilon}$ depends on $p$ and $q$, but we suppress this
dependence in order to avoid cumbersome notations.

\begin{corollary}  \label{nofxhcorollary}
If $X=\sum_{j=0}^{p} x_j e_j+\sum_{j=p+1}^{p+q} \tilde{x}_j\tilde{e_j}
\in \R^{p+q+1}\subset {\mathcal A}_{p,q}\subset {\mathcal A}_{p+q}^\C$, then
\begin{equation} \label{N(h-epsilon)}
N(h_{\epsilon}(X))=(1-\epsilon^2)N(X)+2i\epsilon \norm{X}^2.
\end{equation}
In particular, when $\epsilon=1$,
\begin{equation*}
N(h_1(X))=2i\norm{X}^2.
\end{equation*}
\end{corollary}

\begin{definition}
For $Z_0 \in \C^{p+q+1}$ fixed, we define
\begin{equation*}
h_{\epsilon,Z_0}: Z \quad \mapsto \quad Z_0 + h_{\epsilon}(Z-Z_0).
\end{equation*}
\end{definition}

\begin{theorem}  \label{firstintegralformula}
Let $p+q \geq 2$.
Let $U\subset \R^{p+q+1}\subset {\mathcal A}_{p,q}$ be an open bounded set
with piecewise ${\mathcal C}^1$ boundary $\partial U$, and let
$M_{p+q}^\C$ be a left ${\mathcal A}_{p+q}^\C$-module.
Suppose that an $M_{p+q}^\C$-valued function $f$ is $(p,q)$-left-monogenic
on a neighborhood of the closure $\overline{U}$.
Furthermore, suppose $f$ extends to a complex left-monogenic function
$f^\C: W^\C \to M_{p+q}^\C$, with
$W^\C \subset \C^{p+q+1} \subset {\mathcal A}_{p+q}^\C$
an open subset containing $\overline{U}$. We have:
\begin{equation*}
\int_{(h_{\epsilon,X_0})_*(\partial U)} G_{p+q}(Z-X_0)D_{p+q}z f^\C(Z)
=\begin{cases} \omega_{p+q} f(X_0) & \text{if } X_0 \in U, \\
0 & \text{if } X_0 \notin \overline{U}, \end{cases}
\end{equation*}
for all $\epsilon> 0$ sufficiently close to $0$.

Similarly, if $\tilde{M}_{p+q}^\C$ is a right ${\mathcal A}_{p+q}^\C$-module
and $g$ is an $\tilde{M}_{p+q}^\C$-valued function that is
$(p,q)$-right-monogenic on a neighborhood of $\overline{U}$ and can be extended
to a complex right-monogenic function $g^\C: W^\C \to \tilde{M}_{p+q}^\C$, then
\begin{equation*}
\int_{(h_{\epsilon,X_0})_*(\partial U)} g^\C(Z) D_{p+q}z G_{p+q}(Z-X_0)
=\begin{cases}\omega_{p+q} g(X_0) & \text{if } X_0 \in U, \\
0 & \text{if } X_0 \notin \overline{U}, \end{cases}
\end{equation*}
for all $\epsilon> 0$ sufficiently close to $0$.
\end{theorem}

\begin{remark}
For all $\epsilon \ne 0$ sufficiently close to $0$ the contour of integration
$(h_{\epsilon, X_0})_*(\partial U)$ lies inside $W^{\C}$ and the integrand
is non-singular, thus the integrals are well-defined.
Moreover, we will see that the value of the integral becomes constant when
the parameter $\epsilon$ is sufficiently close to $0$.
\end{remark}

\begin{remark}
  If we require $\epsilon <0$, then we need to insert a factor of $(-1)^q$
  into the right hand sides of the above formulas. Thus,
\begin{equation*}
\int_{(h_{\epsilon,X_0})_*(\partial U)} G_{p+q}(Z-X_0)D_{p+q}z f^\C(Z)
=\begin{cases} (-1)^q \omega_{p+q} f(X_0) & \text{if } X_0 \in U, \\
0 & \text{if } X_0 \notin \overline{U}, \end{cases}
\end{equation*}
for all $\epsilon< 0$ sufficiently close to $0$. 
And similarly for the other formula.
\end{remark}

\begin{proof}
We consider only the left-monogenic case, the right-monogenic case is similar.
By translation, we may assume that $X_0=0$.
Let $M=\sup_{X \in \partial U} \norm{X}$. We restrict $W^{\C}$ to be the
$\delta$-neighborhood of $\overline{U}$ in $\C^{p+q+1}$ for some sufficiently
small $\delta>0$. Consider $0 < |\epsilon| < \delta/M$.
If $X \in \R^{p+q+1} \setminus \{0\}$, then, by \eqref{N(h-epsilon)},
$N(h_{\epsilon}(X))$ is not a negative real,
hence $h_{\epsilon}(X) \in \C_{G}^{p+q+1}$.
Thus, we have that
$(h_{\epsilon})_*(\partial U)$ lies inside $\C_G^{p+q+1} \cap W^\C$.
By Lemma \ref{productrulez}, the integrand is a closed form in the region
$\C_G^{p+q+1} \cap W^\C$, so the integral stays unchanged for all
$-\delta/M < \epsilon < 0$ and $0< \epsilon < \delta/M$
(but the values of the integral may be different on these two intervals).
If $X_0 = 0 \notin \overline{U}$, we have that
$(h_{\epsilon})_*(\overline{U})\subseteq \C_G^{p+q+1} \cap W^{\C}$,
the integrand is a closed form, and the integral over $\partial U$ is zero,
thus we are done.

Now, suppose $X_0 = 0 \in U$.
Choose an $r>0$ sufficiently small such that $r<\delta/2$ and the ball
$B_r = \{X \in \R^{p+q+1}: \norm{X}\leq r\}$ is contained in $U$.
We orient the sphere $S_r=\{X \in \R^{p+q+1}: \norm{X}=r\}$ as
the boundary of $B_r$.
By Stokes' Theorem and \eqref{productrulez1} we have:
\begin{multline*}
\int_{(h_{\epsilon})_*(\partial U)} G_{p+q}(Z) D_{p+q}z f^\C(Z)
- \int_{(h_{\epsilon})_*S_r} G_{p+q}(Z) D_{p+q}z f^\C(Z)  \\
= \int_{(h_{\epsilon})_*(U \setminus B_r)} d(G_{p+q}(Z) D_{p+q}z f^\C(Z))
= \int_{(h_{\epsilon})_*(U \setminus B_r)} 0dV_\C =0.
\end{multline*}
Hence,
\begin{equation*}
\int_{(h_{\epsilon})_*(\partial U)} G_{p+q}(Z) D_{p+q}z f^\C(Z)
= \int_{(h_{\epsilon})_*S_r} G_{p+q}(Z) D_{p+q}z f^\C(Z).
\end{equation*}
By varying the parameter $\epsilon$ from its original value to $1$,
we can continuously deform $(h_{\epsilon})_*S_r$ into $(h_1)_*(S_r)$
within $\C_G^{p+q+1} \cap W^\C$ without changing the value of the integral
(here we are using $r<\delta/2$), and so we have:
\begin{equation*}
\int_{(h_{\epsilon})_*(\partial U)} G_{p+q}(Z) D_{p+q}z f^\C(Z)
= \int_{(h_1)_*S_r} G_{p+q}(Z) D_{p+q}z f^\C(Z).
\end{equation*}

Next, we rotate the sphere $(h_1)_*S_r$ using the map $Z \mapsto e^{-it}Z$,
$0 \leq t \leq \pi/4$. Since $N(e^{-it}Z) = e^{-2it} N(Z)$ and
$\norm{e^{-it}Z} = \norm{Z}$, as $t$ varies continuously from $0$ to $\pi/4$,
the sphere $(h_1)_*S_r$ gets rotated inside $\C_G^{p+q+1} \cap W^\C$
into $\tilde{S}_{r\sqrt{2}}$ -- the sphere of radius $r\sqrt{2}$ contained in
$\operatorname{Span}\{e_0,e_1,\dots,e_{p+q}\}
= \R^{p+q+1} \subset {\mathcal A}_{p+q,0}$.
Thus, we have by Stokes' Theorem and the classical integral formula for
the definite case (Theorem \ref{standardintegralformula}):
\begin{multline*}
\int_{(h_{\epsilon})_*(\partial U)} G_{p+q}(Z)D_{p+q}z f^\C(Z)
=\int_{\tilde{S}_{r\sqrt{2}}} G_{p+q}(Z)D_{p+q}z f^\C(Z) \\
=\int_{\tilde{S}_{r\sqrt{2}}} G_{p+q,0}(X)D_{p+q,0}x f^\C(X)
=\omega_{p+q}f(0).
\end{multline*}
\end{proof}

\section{Second Cauchy's Integral Formula}  \label{6}

In this section we state and prove another integral formula for
$(p,q)$-monogenic functions.
Unlike the previous version, this one that does not change the contour of
integration. Instead, we insert a purely imaginary term $i\epsilon\|X-X_0\|^2$
into the denominator of the reproducing kernel $G_{p,q}(X-X_0)$, then take
limit as $\epsilon \to 0^+$.

\begin{theorem}  \label{secondintegralformula}
Let $p+q \geq 2$.
Let $U\subset \R^{p+q+1}$ be a bounded open region with smooth boundary
$\partial U$, and let $M_{p,q}$ be a left ${\mathcal A}_{p,q}$-module.
Suppose that $f$ is an $M_{p,q}$-valued function that is
$(p,q)$-left-monogenic in a neighborhood of the closure $\overline{U}$.
Then, for any $X_0 \in \R^{p+q+1}$ such that $\partial U$ intersects the cone
$\{ X \in \R^{p+q+1} ;\: N(X-X_0)=0\}$ transversally, we have:
\begin{equation*}
\lim_{\epsilon \to 0^+} \int_{\partial U} G_{p,q,\epsilon}(X-X_0) D_{p,q}x f(X)
= \begin{cases} \omega_{p+q}f(X_0) & \text{if } X_0\in U, \\
  0 & \text{if } X_0 \notin \overline{U}, \end{cases}
\end{equation*}
where $G_{p,q,\epsilon}$ is the modified Green's function defined by
\begin{multline*}
G_{p,q,\epsilon}(X)
= \frac{X^+}{\bigl( N(X)+i\epsilon\norm{X}^2 \bigr)^{\frac{p+q+1}2}}  \\
= \frac{x_0-\sum_{j=1}^p x_j e_j-\sum_{j=p+1}^{p+q} \tilde{x}_j \tilde{e}_j}
{\bigl( \sum_{j=0}^p x_j^2-\sum_{j=p+1}^{p+q} \tilde{x}_j^2
+ i\epsilon(\sum_{j=0}^p x_j^2+\sum_{j=p+1}^{p+q} \tilde{x}_j^2)
\bigr)^{\frac{p+q+1}2}}.
\end{multline*}

Similarly, if $\tilde{M}_{p,q}$ is a right ${\mathcal A}_{p,q}$-module and
$g$ is an $\tilde{M}_{p,q}$-valued function that is
$(p,q)$-right-monogenic in a neighborhood of $\overline{U}$, then,
for any $X_0 \in \R^{p+q+1}$ such that $\partial U$ intersects the cone
$\{ X \in \R^{p+q+1} ;\: N(X-X_0)=0\}$ transversally, we have:
\begin{equation*}
\lim_{\epsilon \to 0^+} \int_{\partial U} g(X) D_{p,q}x G_{p,q,\epsilon}(X-X_0)
=\begin{cases} \omega_{p+q}g(X_0) & \text{if } X_0\in U, \\
0 & \text{if } X_0\notin \overline{U}. \end{cases}
\end{equation*}
\end{theorem}

\begin{remark}
As in the case of the first Cauchy's integral formula
(Theorem \ref{firstintegralformula}),
if we let $\epsilon \to 0^-$, we need to insert a factor of $(-1)^q$
into the right hand sides of the above formulas. Thus,
\begin{equation*}
\lim_{\epsilon \to 0^-} \int_{\partial U} G_{p,q,\epsilon}(X-X_0) D_{p,q}x f(X)
= \begin{cases} (-1)^q \omega_{p+q}f(X_0) & \text{if } X_0\in U, \\
  0 & \text{if } X_0 \notin \overline{U}. \end{cases}
\end{equation*}
And similarly for the other formula.
\end{remark}

\begin{proof}
We change coordinates to hybrid spherical coordinates given by
\begin{equation*}
\Phi(\rho,\phi_1,\dots,\phi_p,\theta,\psi_1,\dots,\psi_{q-1})
=(x_0,x_1,\dots,x_p,\tilde{x}_{p+1},\dots,\tilde{x}_{p+q}),
\end{equation*}
\begin{equation}  \label{sphercoors}
\begin{array}{lcc}
x_0=\rho\cos\theta \cos\phi_1 & \qquad & \\
x_1=\rho\cos\theta\sin\phi_1\cos\phi_2 & & \\
x_2=\rho\cos\theta\sin\phi_1\sin\phi_2\cos\phi_3 & & \\
\vdots & & \\
x_{p-1}=\rho\cos\theta\sin\phi_1\dots \cos \phi_p & &
0\leq \phi_1,\phi_2,\dots,\phi_{p-1}\leq \pi, \\
x_{p}=\rho\cos\theta\sin\phi_1\dots \sin \phi_p & &
0\leq \psi_1,\psi_2,\dots,\psi_{q-2}\leq \pi, \\
\tilde{x}_{p+1}=\rho\sin\theta \cos\psi_1 & & 0\leq \phi_p,\psi_{q-1}\leq 2\pi, \\
\tilde{x}_{p+2}=\rho\sin\theta \sin\psi_1\cos\psi_2 & &
0 \leq \rho, \quad 0\leq \theta\leq \frac{\pi}{2}. \\
\vdots & & \\
\tilde{x}_{p+q-1}=\rho\sin\theta \sin\psi_1\dots \cos\psi_{q-1} & & \\
\tilde{x}_{p+q}=\rho\sin\theta \sin\psi_1\dots \sin\psi_{q-1} & &
\end{array}
\end{equation}
These are essentially the spherical coordinates for the
$(x_0,\dots,x_p)$-subspace combined with the spherical coordinates
for the $(\tilde{x}_{p+1},\dots,\tilde{x}_{p+q})$-subspace.
In these coordinates, we have:
\begin{align*}
N(X) &= \sum_{j=0}^{p} x_j^2 - \sum_{j=p+1}^{p+q} \tilde{x}_j^2
= \rho^2\cos^2\theta - \rho^2\sin^2\theta = \rho^2\cos(2\theta)
\qquad \text{and } \\
\norm{X}^2 &= \sum_{j=0}^{p} x_j^2 + \sum_{j=p+1}^{p+q} \tilde{x}_j^2
= \rho^2\cos^2\theta + \rho^2\sin^2\theta =\rho^2.
\end{align*}
We thus have that the null cone $\mathcal{N}_{p,q}$ is the set of
$X$ in $\R^{p+q+1}$ such that $\theta=\frac{\pi}{4}$,
and we structure our argument in the vein of \cite{L}.
We use the symmetry of the change of basis matrix with respect to
$p$ and $q-1$ to calculate the determinant by means of block matrices.
Let $S_{n,\alpha}$ be the Jacobian matrix corresponding to the transformation
into the standard $(n+1)$-dimensional spherical coordinates,
$(\rho,\alpha_1,\dots,\alpha_n)\to (x_0,\dots,x_n)\in \R^{n+1}$,
\begin{equation*}
\det(S_{n,\alpha}) =
\rho^n\sin^{n-1} \alpha_1\sin^{n-2}\alpha_2\dots \sin \alpha_{n-1}.
\end{equation*}

\begin{lemma}  \label{sphericaljacobian}
The determinant of the Jacobian matrix of the change of coordinates
\eqref{sphercoors} is given by
\begin{equation*}
\det(D\Phi)=\rho\cos^p\theta\sin^{q-1}\theta \det(S_{p,\phi})\det(S_{q-1,\psi}).
\end{equation*}
\end{lemma}

\begin{proof}
We spell out the Jacobian matrix of the change of coordinates \eqref{sphercoors}
and divide it into blocks as follows:
\begin{equation*}
D\Phi=\begin{pmatrix}
\pdv{x_0}{\rho} & \pdv{x_0}{\phi_1}&\pdv{x_0}{\phi_2}&\dots &
\pdv{x_0}{\phi_p}&\vline&\pdv{x_0}{\theta}& \pdv{x_0}{\psi_1}& \dots &
\pdv{x_0}{\psi_{q-1}}\\
\pdv{x_1}{\rho} & \pdv{x_1}{\phi_1}&\pdv{x_1}{\phi_2}&\dots &
\pdv{x_1}{\phi_p}&\vline&\pdv{x_1}{\theta}& \pdv{x_1}{\psi_1}& \dots &
\pdv{x_1}{\psi_{q-1}}\\
\vdots & \vdots &\vdots&\ddots &\vdots&\vline&\vdots& \vdots& \ddots & \vdots\\
\pdv{x_p}{\rho} & \pdv{x_p}{\phi_1}&\pdv{x_p}{\phi_2}&\dots &
\pdv{x_p}{\phi_p}&\vline&\pdv{x_p}{\theta}& \pdv{x_p}{\psi_1}& \dots &
\pdv{x_p}{\psi_{q-1}}\\
\\
\hline
\\
\pdv{\tilde{x}_{p+1}}{\rho} & \pdv{\tilde{x}_{p+1}}{\phi_1}&
\pdv{\tilde{x}_{p+1}}{\phi_2}&\dots &\pdv{\tilde{x}_{p+1}}{\phi_p}&\vline&
\pdv{\tilde{x}_{p+1}}{\theta}& \pdv{\tilde{x}_{p+1}}{\psi_1}& \dots &
\pdv{\tilde{x}_{p+1}}{\psi_{q-1}}\\
\vdots & \vdots &\vdots&\ddots &\vdots&\vline&\vdots& \vdots& \ddots & \vdots\\
\pdv{\tilde{x}_{p+q}}{\rho} & \pdv{\tilde{x}_{p+q}}{\phi_1}&
\pdv{\tilde{x}_{p+q}}{\phi_2}&\dots &\pdv{\tilde{x}_{p+q}}{\phi_p}&\vline&
\pdv{\tilde{x}_{p+q}}{\theta}& \pdv{\tilde{x}_{p+q}}{\psi_1}& \dots &
\pdv{\tilde{x}_{p+q}}{\psi_{q-1}}
\end{pmatrix}
= \begin{pmatrix} A & B\\ \Gamma & \Delta \end{pmatrix}.
\end{equation*}
We make the following observations about the blocks $A$, $B$, $\Gamma$
and $\Delta$.
First, $A$ is the $(p+1)\times (p+1)$ matrix that is obtained from
$S_{p,\phi}$ by multiplying each entry by $\cos\theta$;
we denote the first column of $A$ by $\vec{a}$.
Next, $B$ is the $(p+1)\times q$ matrix whose first column is
$\vec{b}$, described below, and all other columns are zero.
Similarly, $\Gamma$ is the $q\times (p+1)$ matrix whose first column is
$\vec{c}$, described below, and the rest of the columns are zero.
Finally, $\Delta$ the $q \times q$ matrix obtained from $S_{q-1,\psi}$
by multiplying the first column by $\rho\cos\theta$ and multiplying the
remaining columns by $\sin\theta$; we denote the first column of
$\Delta$ by $\vec{d}$.
\begin{equation*}
A = \cos\theta \cdot S_{p,\phi}, \qquad
\Delta = S_{q-1,\psi} \cdot \begin{pmatrix} \rho\cos\theta & 0 & 0 &\dots \\
0 & \sin\theta & 0 & \dots \\ 0 & 0 & \sin\theta & \dots \\
0 & 0 & 0 & \ddots\end{pmatrix},
\end{equation*}
\begin{equation*}
\vec{a}=\cos\theta \cdot \begin{pmatrix}\cos\phi_1\\
\sin\phi_1\cos\phi_2 \\ \vdots \\ \sin\phi_1\sin\phi_2\dots \cos\phi_p \\
\sin \phi_1\sin\phi_2\dots \sin\phi_p\end{pmatrix}, \qquad
\vec{b}= -\frac{\rho\sin\theta}{\cos\theta} \cdot \vec{a},
\end{equation*}
\begin{equation*}
\vec{c}=\sin\theta \cdot
\begin{pmatrix}\cos\psi_1\\ \sin\psi_1\cos\psi_2 \\ \vdots \\
\sin\psi_1\sin\psi_2\dots \cos\psi_{q-1} \\
\sin \psi_1\sin\psi_2\dots \sin\psi_{q-1}\end{pmatrix}, \qquad
\vec{d}= \frac{\rho\cos\theta}{\sin\theta} \cdot \vec{c}.
\end{equation*}

\begin{lemma}
  If $A$ is an $m\times m$ matrix, $B$ is an $m\times n$ matrix,
  $\Gamma$ is an $n\times m$ matrix, and $\Delta$ is an invertible
  $n\times n$ matrix, then
  $\det\begin{pmatrix} A & B\\ \Gamma & \Delta \end{pmatrix}
  =\det(A-B\Delta^{-1}\Gamma)\det(\Delta)$.
\end{lemma}

\begin{proof}
  This standard result on block matrices is proved by factoring into
  triangular matrices:
\begin{equation*}
  \begin{pmatrix} A & B\\ \Gamma & \Delta \end{pmatrix}
  =\begin{pmatrix} A-B\Delta^{-1}\Gamma & B\Delta^{-1}\\
  0_{n,m} & I_{n,n}\end{pmatrix}
  \begin{pmatrix} I_{m,m} & 0_{m,n}\\ \Gamma & \Delta\end{pmatrix}.
\end{equation*}
\end{proof}

Using this lemma,
\begin{equation*}
\det(D\Phi)=\det\begin{pmatrix} A & B\\ \Gamma & \Delta\end{pmatrix}
=\det(A-B\Delta^{-1}\Gamma)\det(\Delta).
\end{equation*}
We next calculate $A-B\Delta^{-1}\Gamma$.
We first compute $B^{-1}\Delta \Gamma$, labelling the $j$th row of the
matrix $\Delta^{-1}$ by $\vec{r_j}$ (horizontal vector),
\begin{equation*}
B(\Delta^{-1}\Gamma) =B\left(\begin{pmatrix}
\vec{r}_1\\ \vec{r}_2\\ \vdots \\ \vec{r}_{q}\end{pmatrix}
\begin{pmatrix} \vec{c} & \vec{0} &\dots \end{pmatrix}\right)
=\begin{pmatrix} \vec{b}& \vec{0} &\dots\end{pmatrix} \begin{pmatrix} \vec{c}
\cdot \vec{r}_1 & 0 &\dots \\\vec{c}\cdot \vec{r}_2 & 0 &\dots \\ \vdots \\
\vec{c}\cdot \vec{r}_q & 0 &\dots \end{pmatrix}
=\begin{pmatrix} (\vec{c}\cdot \vec{r}_1)\vec{b}& \vec{0} & \dots\end{pmatrix}.
\end{equation*}
Observe that $\vec{c}$ is the first column of $\Delta$ scaled by
$\frac{\sin\theta}{\rho\cos\theta}$, and the so the scalar product of it and
the first row of $\Delta^{-1}$ is $\frac{\sin\theta}{\rho\cos\theta}$,
as $1$ is the upper left entry of $\Delta\Delta^{-1}$.
Thus, we have that $B\Delta^{-1}\Gamma$ is a matrix with first column
$\frac{\sin\theta}{\rho\cos\theta} \vec{b}$, and the rest of columns zero.
Since $\vec{b}= -\frac{\rho\sin\theta}{\cos\theta} \vec{a}$,
we have that $A-B\Delta^{-1}\Gamma$ is matrix $A$ with the first column
multiplied by
$1+\frac{\sin^2\theta}{\cos^2\theta}=\frac{1}{\cos^2\theta}$.
Using the multilinearity of the determinant, and factoring out the
$\cos\theta$ and $\sin\theta$ from the columns of $A$ and $\Delta$,
\begin{multline*}
\det(D\Phi)=\det(A-B\Delta^{-1}\Gamma)\det(\Delta)
=\frac1{\cos^2\theta} \det(A)\det(\Delta)  \\
=\frac1{\cos^2\theta}(\cos^{p+1}\theta \cdot \rho\cos\theta \cdot
\sin^{q-1}\theta) \det(S_{p,\phi})\det(S_{q-1,\psi}) \\
=\rho\cos^{p}\theta\sin^{q-1}\theta\det(S_{p,\phi})\det(S_{q-1,\psi}).
\end{multline*}
\end{proof}

We return to the proof of Theorem \ref{secondintegralformula}.
The case $X_0 \notin \overline{U}$ is easier, so we assume $X_0 \in U$.
By translation we can also assume that $X_0=0$.
This lemma is proved by direct computation.

\begin{lemma}  \label{diracofmodifiedgreensfunction}
We have:
\begin{align*}
\nabla^+_{p,q} G_{p,q,\epsilon} &= i\epsilon (p+q+1) \frac{\norm{X}^2-\bar{X}X^+}
{\bigl( N(X)+i\epsilon \norm{X}^2 \bigr)^{\frac{p+q+3}2}},  \\
G_{p,q,\epsilon}\nabla^+_{p,q} &= i\epsilon (p+q+1) \frac{\norm{X}^2-X^+\bar{X}}
{\bigl( N(X)+i\epsilon \norm{X}^2 \bigr)^{\frac{p+q+3}2}}.
\end{align*}
\end{lemma}

By Lemma \ref{productrulex},
\begin{equation*}
d(G_{p,q,\epsilon}\cdot D_{p,q}x \cdot f)
=(G_{p,q,\epsilon}\nabla^+_{p,q})fdV_{p,q}
+G_{p,q,\epsilon}(\nabla^+_{p,q} f) dV_{p,q}=(G_{p,q,\epsilon}\nabla^+_{p,q}) fdV_{p,q},
\end{equation*}
since $f$ is $(p,q)$-left-monogenic.
Then by Stokes' theorem we have
\begin{multline}  \label{twointegrals}
\int_{\partial U}G_{p,q,\epsilon}(X)\cdot D_{p,q}x\cdot f(X)  \\
= \int_{U\setminus B_r}\frac{i\epsilon (p+q+1)(\norm{X}^2-X^{+}\bar{X})}
{\bigl( N(X)+i\epsilon \norm{X}^2 \bigr)^{\frac{p+q+3}2}}f(X)dV_{p,q}
+ \int_{S_r}G_{p,q,\epsilon}(X)\cdot D_{p,q}x\cdot f(X),
\end{multline}
where $B_r = \{ X \in \R^{p+q+1}; \norm{X} \leq r \}$ is a ball of sufficiently
small radius $r$ centered at the origin and $S_r$ is its boundary sphere.

Next, we establish (in order) analogues of Lemma 17, Lemma 18, and Lemma 20
of \cite{L}.
Recall that we have selected a branch of $z^{1/2}$ with values in the right
half-plane, as discussed after Definition \ref{defthatprecedessquareroot}.

\begin{lemma}  \label{distributionlemma}
Fix a $\theta_0\in (0,\frac{\pi}{4})$, and let $p,q$ be non-negative integers,
then we have two distributions which send a test function $g(\theta)$ into
the limits
\begin{equation*}
\lim_{\epsilon\to 0^+} \int_{\frac{\pi}{4}-\theta_0}^{\frac{\pi}{4}+\theta_0}
\frac{g(\theta)d\theta}{\bigl( \cos(2\theta)+i\epsilon \bigr)^{\frac{p+q+3}2}}
\qquad \text{and} \qquad
\lim_{\epsilon\to 0^-} \int_{\frac{\pi}{4}-\theta_0}^{\frac{\pi}{4}+\theta_0}
\frac{g(\theta)d\theta}{\bigl( \cos(2\theta)+i\epsilon \bigr)^{\frac{p+q+3}2}}.
\end{equation*}
In other words, the above two limits of integrals define continuous linear
functionals on the topological vector space of smooth functions\footnote{
Note that this vector space has the standard Frech\'et toplogy determined
by the seminorms
$\max_{\theta \in [\frac{\pi}{4}-\theta_0, \frac{\pi}{4}-\theta_0]} g^{(k)}(\theta)$,
where $k=0,1,2,\dots$.}
on the interval $[\frac{\pi}{4}-\theta_0, \frac{\pi}{4}-\theta_0]$.
\end{lemma}

\begin{proof}
In the case of $p+q=1\mod 2$, this is a consequence of Lemma $17$ of \cite{L}.
Otherwise, we modify the proof to fit the fractional case.
We have that $(p+q+3)/2 = n+1/2$ for some non-negative integer $n$,
and we use induction on $n$. For the base case, $n=0$, we integrate by parts,
\begin{multline*}
\int_{\frac{\pi}{4}-\theta_0}^{\frac{\pi}{4}+\theta_0}
\frac{g(\theta)d\theta}{\bigl(\cos(2\theta)+i\epsilon\bigr)^{1/2}}
=\int_{\frac{\pi}{4}-\theta_0}^{\frac{\pi}{4}+\theta_0}
\frac{\sin(2\theta)}{\bigl(\cos(2\theta)+i\epsilon\bigr)^{1/2}}
\frac{g(\theta)d\theta}{\sin(2\theta)}\\
=-\eval{ \bigl(\cos(2\theta)+i\epsilon\bigr)^{1/2}
\frac{g(\theta)}{\sin(2\theta)}}_{\pi/4-\theta_0}^{\pi/4+\theta_0}
+\int_{\frac{\pi}{4}-\theta_0}^{\frac{\pi}{4}+\theta_0}
\bigl(\cos(2\theta)+i\epsilon\bigr)^{\frac12}
\dv{\theta}(\frac{g(\theta)}{\sin(2\theta)})d\theta.
\end{multline*}
As $\bigl(\cos(2\theta)+i\epsilon\bigr)^{1/2}$ is integrable for all values
of $\epsilon$, including $\epsilon=0$, the limits as $\epsilon\to 0^{\pm}$
exist and depend continuously on $g(\theta)$. 
Now we consider the case of $n>0$, in which case we can integrate by parts,
\begin{multline*}
\int_{\frac{\pi}{4}-\theta_0}^{\frac{\pi}{4}+\theta_0}
\frac{g(\theta)d\theta}{\bigl(\cos(2\theta)+i\epsilon\bigr)^{n+1/2}}
=\int_{\frac{\pi}{4}-\theta_0}^{\frac{\pi}{4}+\theta_0}
\frac{\sin(2\theta)}{\bigl(\cos(2\theta)+i\epsilon\bigr)^{n+1/2}}
\frac{g(\theta)d\theta}{\sin(2\theta)}\\
=\eval{\frac{\bigl(\cos(2\theta)+i\epsilon\bigr)^{1/2-n}}{2n-1}
\frac{g(\theta)}{\sin(2\theta)}}_{\pi/4-\theta_0}^{\pi/4+\theta_0}
- \int_{\frac{\pi}{4}-\theta_0}^{\frac{\pi}{4}+\theta_0} \frac{(2n-1)^{-1}}
{\bigl(\cos(2\theta) +i\epsilon\bigr)^{n-\frac12}}
\dv{\theta}(\frac{g(\theta)}{\sin(2\theta)})d\theta,
\end{multline*}
and the result follows by induction on $n$. 
\end{proof}

\begin{lemma}  \label{exteriorderivativetozero}
\begin{equation*}
\lim_{\epsilon \to 0}\int_{U\setminus B_r}
\frac{i\epsilon(p+q+1)(\norm{X}^2-X^+\bar{X})}
{\bigl(N(X)+i\epsilon \norm{X}^2\bigr)^{\frac{p+q+3}2}} f(X) dV_{p,q}=0.
\end{equation*}
\end{lemma}

\begin{proof}
We write the integral in the hybrid spherical coordinates \eqref{sphercoors}
and integrate out the variables $r,\phi_1,\dots,\phi_n,\psi_p,\dots,\psi_{q-1}$.
After we do this, we retain an integral of the form
\begin{equation}  \label{similarintegral}
\epsilon \int _0^{\frac{\pi}2} \frac{g(\theta) d\theta}
{\bigl(\cos(2\theta)+i\epsilon\bigr)^{\frac{p+q+3}{2}}},
\end{equation}
for some function $g(\theta)$.
Due to the transversality of the boundary $\partial U$ with respect to the
null cone ${\cal N}_{p,q} = \{ \theta= \frac{\pi}4 \}$,
the function $g(\theta)$ is smooth at least for
$\theta$ lying in some interval $[\frac{\pi}4-\theta_0,\frac{\pi}4+\theta_0]$
with $\theta_0\in (0,\frac{\pi}4)$.
We can apply Lemma \ref{distributionlemma} to
$\int_{\frac{\pi}{4}-\theta_0}^{\frac{\pi}{4}+\theta_0} \dots$,
and the limit of the integral on the remainder of the interval exists,
so the limit of \eqref{similarintegral} as $\epsilon \to 0$ is zero.
\end{proof}

We introduce a constant $C_{p,q}$
(it will be evaluated later in Lemma \ref{C-lemma}):
\begin{equation}  \label{C_{p,q}def}
C_{p,q}=\lim_{\epsilon\to 0^+} \int_0^{\frac{\pi}2}
\frac{\cos^p\vartheta \sin^{q-1}\vartheta}
{\bigl(\cos (2\vartheta)+i\epsilon\bigr)^{\frac{p+q+1}2}} d\vartheta.
\end{equation}
Note that this limit exists by Lemma \ref{distributionlemma}.

\begin{lemma}  \label{firstintegralequation}
\begin{equation*}
\lim_{r\to 0^+} \biggl( \lim_{\epsilon \to 0^+} \int_{S_r}
\frac{r f(X) dS_{p,q}}{\bigl(N(X)+i\epsilon r^2\bigr)^{\frac{p+q+1}{2}}} \biggr)
= i^q\omega_p\omega_{q-1}C_{p,q}f(0).
\end{equation*}
\end{lemma}

\begin{proof}
By Lemma \ref{sphericaljacobian}, we can split
\begin{equation*}
dS_{p,q}= i^q r^{p+q}\cos^p\theta \sin^{q-1}\theta
d\theta d\Omega_{p,\phi}d\Omega_{q-1,\psi},
\end{equation*}
where
\begin{equation*}
d\Omega_{n,\alpha}= \rho^{-n} \det(S_{n,\alpha}) d\alpha_1 d\alpha_2\dots d\alpha_n
\end{equation*}
represents the spherical volume element on the $n$-dimensional sphere.
The factor of $i^q$ comes from the
normalization of the Euclidean volume element on this space, \eqref{factorsofi}.
Note that the factors of $r$ in the numerator and denominator of
\begin{equation*}
\frac{r f(X) dS_{p,q}}{\bigl(N(X)+i\epsilon r^2\bigr)^{\frac{p+q+1}2}}
\end{equation*}
cancel.
Let $F_{\epsilon}(\theta)$ be the following antiderivative:
\begin{equation*}
F_{\epsilon}(\theta)=\int_0^{\theta} \frac{\cos^p\vartheta \sin^{q-1}\vartheta}
{\bigl(\cos(2\vartheta)+i\epsilon\bigr)^{\frac{p+q+1}2}} d\vartheta.
\end{equation*}
Integrating by parts yields
\begin{multline*}
\int_{S_r} \frac{r^{-p-q} f(X) dS_{p,q}}
{\bigl(\cos(2\theta)+i\epsilon\bigr)^{\frac{p+q+1}2}}
= i^q \int_{S_{q-1}}\int_{S_p}\int_{\theta=0}^{\theta=\pi/2}
\frac{f(X) \cos^p\theta \sin^{q-1}\theta d\theta d\Omega_{p,\phi}d\Omega_{q-1,\psi}}
{\bigl(\cos(2\theta)+i\epsilon\bigr)^{\frac{p+q+1}2}}  \\
=i^q\int_{S_{q-1}}\int_{S_p}\eval{f F_{\epsilon}(\theta)}_{\theta=0}^{\theta=\pi/2}
d\Omega_{p,\phi}d\Omega_{q-1,\psi}
-i^q\int_{S_{q-1}}\int_{S_p}F_{\epsilon}(\theta)\pdv{f}{\theta}
d\theta d\Omega_{\phi,p}d\Omega_{\psi,q-1}.
\end{multline*}
By the chain rule, we have
$\pdv{f}{\theta}=\sum_i \pdv{f}{x_i}\pdv{x_i}{\theta}=r\cdot g(X)$,
where $g(X)$ is a smooth function.
Thus, the second term is bounded by a constant multiple of $r$
(by Lemma \ref{distributionlemma}), and so in the limit of $r\to 0^+$
it vanishes.
In the first term, as $r \to 0^+$, $f$ approaches $f(0)$,
and so we can carry it out, and integrate $d\Omega_{p,\phi}$ and
$d\Omega_{q-1,\psi}$ to be the respective volumes of unit spheres to obtain
\begin{multline*}
\lim_{r\to 0^+} \biggl( \lim_{\epsilon \to 0^+}
i^q\int_{S_{q-1}}\int_{S_p}\eval{f F_{\epsilon}(\theta)}_{\theta=0}^{\theta=\pi/2}
d\Omega_{p,\phi}d\Omega_{q-1,\psi} \biggr)  \\
= f(0)\omega_{p}\omega_{q-1}\lim_{\epsilon\to 0^+}
\eval{F_{\epsilon}(\theta)}_{\theta=0}^{\theta=\pi/2}
=i^q\omega_p\omega_{q-1} C_{p,q}f(0).
\end{multline*}
\end{proof}

The following integrals appear as certain cross terms.

\begin{lemma}  \label{secondintegralequation}
For each $0 \leq j \leq p$ and $1 \leq k \leq q$,
\begin{equation*}
\lim_{r\to 0^+} \biggl( \lim_{\epsilon\to 0^+} \int_{S_r}
\frac{x_j \tilde{x}_{p+k} f(X)}
{\bigl(N(X)+i\epsilon r^2\bigr)^{\frac{p+q+1}2}}\frac{dS_{p,q}}r \biggr) =0.
\end{equation*}
\end{lemma}

\begin{proof}
Note that, as in the proof of previous lemma, once we convert to the hybrid
spherical coordinates \eqref{sphercoors}, the factors of $r$ cancel.
After conversion to these coordinates, the integrand depends on $\psi_k$
as follows.
For $k=1,\dots,q-1$, the integrand is of the form $\cos\psi_k \cdot h_k$,
where $h_k$ is a function does not depend on $\psi_k$.
And for $k=q$, it is of the form $\sin\psi_{q-1} \cdot h_q$.
Additionally, we see from Lemma \ref{sphericaljacobian} that the Jacobian,
as a function of $\psi_k$, is proportional to $\sin^{q-1-k} \psi_k$
for $k=1,\dots,q-1$.
By absorbing the remaining factors of
$x_j(\cos(2\theta)+i\epsilon)^{-\frac{p+q+1}2} dS_{p,q}$ into the $h_k$,
so we still have $\pdv{h_k}{\psi_k}=0$, for all $\epsilon$,
we can rewrite the integral as
\begin{equation*}
\begin{cases}
\int_{S_r} \cos\psi_k \sin^{q-1-k} \psi_kh_kf(X)
d\theta d\phi_1d\phi_2\dots d\phi_p d\psi_1 d\psi_2\dots d\psi_{q-1}
& \text{if $k=1,\dots,q-2$}, \\
\int_{S_r} \cos\psi_{q-1} h_{q-1}f(X)
d\theta d\phi_1d\phi_2\dots d\phi_p d\psi_1 d\psi_2\dots d\psi_{q-1}
& \text{if $k=q-1$}, \\
\int_{S_r} \sin\psi_{q-1}h_{q}f(X)
d\theta d\phi_1d\phi_2\dots d\phi_p d\psi_1 d\psi_2\dots d\psi_{q-1}
& \text{if $k=q$}.
\end{cases}
\end{equation*}
For $1\leq k\leq q-2$, when we integrate with respect to $\psi_k$, we have
\begin{multline*}
\int_{\psi_k=0}^{\psi_k=\pi}  \cos\psi_k \sin^{q-1-k}\psi_k h_kf(X) d\psi_k  \\
=\eval{\frac{\sin^{q-k}\psi_k}{q-k} h_kf(X)}_{\psi_k=0}^{\psi_k=\pi}
-\int_0^{\pi}\frac{\sin^{q-k}\psi_k}{q-k} h_k\pdv{f}{\psi_k} d\psi_k
= -\int_0^{\pi}\frac{\sin^{q-k}\psi_k}{q-k} h_k\pdv{f}{\psi_k} d\psi_k.
\end{multline*}
The partial derivative $\pdv{f}{\psi_k}$ yields a factor of $r$ by the chain
rule, as demonstrated in the proof of Lemma \ref{firstintegralequation}.
Then we integrate out the remaining variables and let $\epsilon \to 0^+$,
the limit exists by Lemma \ref{distributionlemma}.
As $r \to 0^+$, the factor of $r$ ensures the limit is zero.

For $k=q-1$ and $k=q$, the bounds of integration are
$0 \leq \psi_{q-1} \leq 2\pi$,
and the same argument works in this case as well. 
\end{proof}

Now  we can prove an analogue of Lemma 20 of \cite{L}.

\begin{lemma}  \label{lastlemma}
Let $C_{p,q}$ be as in \eqref{C_{p,q}def}, then
\begin{equation*}
\lim_{r\to 0^+} \biggl( \lim_{\epsilon\to 0^+} \int_{S_r}
\frac{X^+ D_{p,q}x f(X)}{\bigl(N(X)+i\epsilon r^2\bigr)^{\frac{p+q+1}2}} \biggr)
=i^q\omega_{p}\omega_{q-1}C_{p,q}f(0).
\end{equation*}
\end{lemma}

\begin{proof}
Using Corollary \ref{dSlemma}, we have
\begin{equation}  \label{analogequation1}
\int_{S_r} \frac{X^+ D_{p,q}x f(X)}{\bigl(N(X)+i\epsilon r^2\bigr)^{\frac{p+q+1}2}}
= \int_{S_r} \frac{X^+\bar{X} f(X)}{\bigl(N(X)+i\epsilon r^2\bigr)^{\frac{p+q+1}2}}
\frac{dS_{p,q}}r.
\end{equation}
We calculate, if
$X=x_0e_0+\sum_{j=1}^p x_je_j+\sum_{k=p+1}^{p+q} \tilde{x}_k\tilde{e}_k$,
\begin{equation*}
X^{+}\bar{X}
=\norm{X}^2-2\sum_{j=0}^p \sum_{k=p+1}^{p+q} x_j\tilde{x}_k\tilde{e}_k e_j.
\end{equation*}
We split the integral of \eqref{analogequation1} into two parts
to be analyzed separately,
\begin{multline*}
\int_{S_r} \frac{X^+ D_{p,q}x f(X)}{\bigl(N(X)+i\epsilon r^2\bigr)^{\frac{p+q+1}2}}
=\int_{S_r} \frac{\norm{X}^2 f}{\bigl(N(X)+i\epsilon r^2\bigr)^{\frac{p+q+1}2}}
\frac{dS_{p,q}}r  \\
-2\sum_{j=0}^p \sum_{k=p+1}^{p+q} \int_{S_r}
\frac{x_j\tilde{x}_k \tilde{e}_ke_j f(X)}
{\bigl(N(X)+i\epsilon r^2\bigr)^{\frac{p+q+1}2}}\frac{dS_{p,q}}r.
\end{multline*}
Applying Lemmas \ref{firstintegralequation} and \ref{secondintegralequation}
yields the result.
\end{proof}

We have proved Theorem \ref{secondintegralformula} up to a constant coefficient.
Indeed, combining \eqref{twointegrals} with
Lemmas \ref{exteriorderivativetozero} and \ref{lastlemma}, we obtain
\begin{equation}  \label{2nd-formula-constant}
\lim_{\epsilon \to 0^+} \int_{\partial U} G_{p,q,\epsilon}(X-X_0) D_{p,q}x f(X)
= i^q\omega_{p}\omega_{q-1}C_{p,q}f(X_0) \qquad \text{if $X_0 \in U$}.
\end{equation}
And if $X_0 \notin \overline{U}$ the same argument shows that the limit of
integrals is zero.
Then Theorem \ref{secondintegralformula} follows from Lemma \ref{C-lemma} below.
\end{proof} 

\begin{lemma}  \label{C-lemma}
\begin{equation*}
C_{p,q}=(-i)^q\frac{\omega_{p+q}}{\omega_p\omega_{q-1}}.
\end{equation*}
\end{lemma}
\begin{proof}
We prove this by applying both integral formulas to the constant function
$f(X)=1$ on the sphere $S_{p,q}=\{X\in \R^{p+q+1}: \norm{X}=1\}$.
When we apply the first integral formula (Theorem \ref{firstintegralformula}),
we obtain, for all $\epsilon>0$ sufficiently close to zero,
\begin{equation*}
\int_{(h_{\epsilon})_*(S_{p,q})} \frac{Z^+ D_{p+q}z}{N(Z)^{\frac{p+q+1}2}}=\omega_{p+q}.
\end{equation*}
We calculate the pullback of $D_{p+q}z$ and rewrite this as an integral over
the sphere $S_{p,q}$.
Note that $h_{\epsilon}$ is a linear transformation and that
\begin{equation*}
(h_{\epsilon})^* dz_j
= \begin{cases} (1+i\epsilon)dz_j & \text{if $0 \leq j \leq p$},  \\
(1-i\epsilon)dz_j & \text{if $p+1 < \leq j \leq p+q$}. \end{cases}
\end{equation*}
Using the expansion of $D_{p+q}z$ in the standard basis \eqref{Dzsumdef},
we find:
\begin{multline*}
(h_{\epsilon})^* D_{p+q}z
= \sum_{j=0}^p (-1)^j e_j (1+i\epsilon)^p(1-i\epsilon)^q d\hat{z}_j
+ \sum_{j=p+1}^{p+q} (-1)^j e_j (1+i\epsilon)^{p+1}(1-i\epsilon)^{q-1} d\hat{z}_j\\
=D_{p+q}z + \epsilon \tilde{D}_{p+q,\epsilon}z,
\end{multline*}
where the form $\tilde{D}_{p+q,\epsilon}z$ depends on $\epsilon$ polynomially
and does not depend on $Z$.
Let $\tilde{D}_{p,q,\epsilon}x$ be the restriction of $\tilde{D}_{p+q,\epsilon}z$
to $\R^{p+q+1}$.
Recall the conjugation \eqref{compl-conj-Rn}.
If $X \in \R^{p+q+1}$ and $Z=h_{\epsilon}(X)$, then, using \eqref{N(h-epsilon)},
\begin{equation*}
(h_{\epsilon})^* Z^+=(X+i\epsilon \bar{X})^+ = X^++i\epsilon \bar{X}^+,
\end{equation*}
\begin{equation*}
(h_{\epsilon})^* N(Z)=(1-\epsilon^2)N(X)+i\epsilon\norm{X}^2
= (1-\epsilon^2)N(X)+i\epsilon.
\end{equation*}
We have:
\begin{multline*}  
\int_{(h_{\epsilon})_*(S_{p,q})} \frac{Z^+ D_{p+q}z}{N(Z)^{\frac{p+q+1}2}}
=\int_{S_{p,q}} \frac{(X^++i\epsilon \bar{X}^+)
(D_{p,q}x+\epsilon\tilde{D}_{p,q,\epsilon}x)}
{\bigl((1-\epsilon^2)N(X)+i\epsilon\bigr)^{\frac{p+q+1}2}}  \\
= \int_{S_{p,q}} \frac{X^+D_{p,q}x}
{\bigl((1-\epsilon^2)N(X)+i\epsilon\bigr)^{\frac{p+q+1}2}}
+ \int_{S_{p,q}} \frac{\epsilon \bar{X}^+ \tilde{D}_{p,q,\epsilon}x
+ i\epsilon \bar{X}^+(D_{p,q}x+\epsilon \tilde{D}_{p,q,\epsilon}x)}
{\bigl((1-\epsilon^2)N(X)+i\epsilon\bigr)^{\frac{p+q+1}2}}.
\end{multline*}
As $\epsilon \to 0^+$, the first integral approaches
$i^q\omega_p\omega_{q-1}C_{p,q}$ by the already established formula
\eqref{2nd-formula-constant}.
And the second integral approaches zero because we can factor out $\epsilon$
and proceed in the same manner as we proved Lemma \ref{exteriorderivativetozero}
using the hybrid coordinates \eqref{sphercoors} and
Lemma \ref{distributionlemma}.
\end{proof}

\end{document}